\documentclass[reqno]{amsart}

\usepackage[foot]{amsaddr} 

\usepackage[pagewise]{lineno}\nolinenumbers

\numberwithin{equation}{section}
\usepackage{latexsym}
\usepackage{amsmath,hyperref}
\usepackage{mathscinet}
\usepackage{enumitem}
\usepackage{amssymb}
\usepackage{mathrsfs}
\usepackage{graphicx,colordvi,psfrag}
\usepackage{upgreek}
\usepackage{ifthen}
\usepackage{color}

\usepackage[T1]{fontenc}
\usepackage[latin1]{inputenc}

\numberwithin{equation}{section}

\newtheorem{definition}{Definition}[section]
\newtheorem{theorem}[definition]{Theorem}
\newtheorem{lemma}[definition]{Lemma}

\newtheorem{remark}[definition]{Remark}

\usepackage{latexsym}
\usepackage{amsmath}
\usepackage{amssymb}

\newcommand{\R}{{\mathbb R}}
\newcommand{\N}{{\mathbb N}}

\providecommand{\norm}[1]{\| #1 \|}

\allowdisplaybreaks 

\title[]{Continuous dependence results for quasilinear evolution equations}
\author[F. Cellarosi]{Francesco Cellarosi}
\author[A. Dutta]{Anirban Dutta}
\author[G. Mazzone]{Giusy Mazzone}
\date{\today}
\address{Queen's University}
\email{francesco.cellarosi@queensu.ca}
\email{21ad53@queensu.ca}
\email{giusy.mazzone@queensu.ca}

\begin{document}



\begin{abstract}
We study continuous dependence of solutions to quasilinear evolution equations of parabolic-type in the framework of maximal $L^p$-regularity. For equations of the form
\[
\frac{d\phi}{dt} + A(t,\phi)\phi = f(t,\phi),
\]
we establish continuous dependence of strong solutions on initial data, 
and suitable approximations of the nonlinear operators $A$ and $f$. An important step for proving the main result is the fact that the maximal regularity constant of the operator $A(t,\phi)$, with $t$ and $\phi$ fixed, admits a uniform bound over compact subsets of the relevant Banach spaces. 

As an application, we consider a class of non-Newtonian fluid models with  a Carreau-type viscosity and mixed boundary conditions. We show that, as the nonlinear contribution in the viscosity vanishes and  the initial data converge, solutions of the non-Newtonian  fluid model converge to those of the classical Navier--Stokes equations.
\end{abstract}

\maketitle 

\section{Introduction}
The abstract theory of evolution equations of parabolic-type provides a powerful and flexible framework for the analysis of nonlinear time-dependent partial differential equations. Over the past decades, it has become one of the central tools for studying local and global well-posedness, continuous dependence of solutions on  parameters, and the construction of the associated semiflow in suitable phase spaces. Beyond existence and uniqueness, the same abstract approach allows for a systematic investigation of qualitative properties of the solutions, such as stability of equilibria, the structure of invariant manifolds near steady states, and the long-time asymptotic behavior of trajectories, including convergence to equilibrium (see \cite{Prussbook,PSZ09,Simonett}).

In this paper, we consider a quasilinear parabolic-type problem of the form
\begin{equation}\label{eq:main_problem}
\begin{cases}
\displaystyle
\frac{d\phi}{dt} + A(t,\phi)\phi = f(t,\phi), & t > t_{0}, \\[8pt]
\phi(t_{0}) = u_{0}\in X.
\end{cases}
\end{equation}
Here, for each fixed $t$ and $\phi$, the operator $A(t,\phi)$ is a (possibly unbounded) linear  operator on $X$,
while $f$ is a nonlinear operator on $X$. 
We refer to Section~\ref{sec:notation} for the precise statement of the assumptions imposed on $A$, $f$, and the initial datum $u_0$ (see conditions \ref{assumption_structure}--\ref{existence_reason4}). Under these assumptions, the local-in-time existence of strong solutions to \eqref{eq:main_problem} is already known; see, for instance, \cite[Theorem~3.1]{Bari}.

The main result of our paper is a proof of the continuous dependence of solutions to \eqref{eq:main_problem} upon the initial data, and suitable approximations of the operators $A$ and $f$, in the framework of maximal $L^p$-regularity (see Theorem \ref{main theorem1}). 
Our result extends considerably existing results obtained by K\"ohne, Pr\"uss, and Wilke in \cite{MR2643804} for the quasilinear equation \eqref{eq:main_problem}, where only continuous dependence upon the initial data was considered. 
In the context of semilinear equations, i.e., when the operator $A$ in \eqref{eq:main_problem} does not depend on the solution, a similar continuous dependence result (upon data and forces) was obtained by the second author of this paper in \cite{anirban}.

A key ingredient of our analysis is an auxiliary result on the uniform boundedness of the maximal regularity constants for the operator $A(t,u)$ for each fixed $t$ and $u$ (see Lemma~\ref{Introductory_lemma}). In general, the maximal regularity constant associated with $A(t,u)$ may depend on both $t$ and $u$. We prove that, if $t$ varies in a bounded subset of $[t_0,\infty)$ and $u$ varies in a  compact subset of a suitable Banach space, then these maximal regularity constants remain uniformly bounded.

Beyond its intrinsic analytical interest, the abstract framework considered here has a wide range of applications. Numerous  mathematical models arising in physical applications  
can be reformulated as quasilinear parabolic-type evolution equations. These include models for viscous, incompressible, and non-Newtonian fluids, Stefan problems with surface tension, two-phase Navier--Stokes systems, and geometric evolution equations such as the mean curvature flow, the surface diffusion flow, and the Willmore flow  (see e.g.\cite{intro_pruss}, \cite{Prussbook}, \cite{intro_simonett}). Continuous dependence results on data, forcing and parameters provide valuable information on the mathematical models as well as on the physical description of their solutions. 
Our abstract approach thus provides a unifying way 
to treat different physical problems within a common functional analytic setting.

In Section~\ref{application}, we apply our abstract continuous dependence result to a class of generalized Newtonian models describing incompressible, viscous, non-Newtonian fluid flows arising in hemodynamics. These non-Newtonian fluids are characterized by the following constitutive equation for the Cauchy stress tensor: 
\[
\mathcal{S}_n(v,\pi)
= 2 \mu_n(|D(v)|^2) D(v) - \pi \mathbf{I},
\]
where $D(v) = \frac{1}{2}\big(\nabla v + (\nabla v)^T\big)$ is the rate of deformation tensor, and the viscosity coefficient $\mu_n$ satisfies the {\em Carreau-type law}
\[
\mu_n(s)= \mu_\infty + \eta_n (1+s)^{\frac{d-2}{2}},\qquad \text{for }s \ge 0,
\]
with $\mu_\infty>0$ a constant and $\{\eta_n\}_{n\in\mathbb{N}}$ a sequence of non-negative real numbers such that $\eta_n \to 0$ as $n\to\infty$.

In the special case $\eta_n \equiv 0$, the viscosity becomes constant and we retrieve the usual constitutive equation of a viscous, incompressible, Newtonian fluid. The corresponding equations of motion would then coincide with the classical incompressible Navier--Stokes equations. 

On the boundary, we impose mixed boundary conditions. Specifically, we impose conditions of {\em no-slip} and {\em pure slip} on different portions of the fluid's boundary.

We prove that, as $\eta_n \to 0$ and the initial fluid velocity $v_0^n$ converge (in an appropriate topology),
solutions of the non-Newtonian fluid model (see equation \eqref{eq:motion}) converge to the solution of the Navier--Stokes system in the setting of maximal $L^p$-regularity. The proof relies on reformulating the fluid equations as  a quasilinear evolution equation in an appropriate Banach space setting (see equations \eqref{eq:motion1} and \eqref{evolution-equation-application}), and then applying our abstract continuous dependence result given in Theorem \ref{main theorem1}.

\section{Notation and Preliminaries}\label{sec:notation}

Let \( X_0 \) and \( X_1 \) be Banach spaces such that the embedding \( X_1 \hookrightarrow X_0 \) is continuous and the image of \( X_1 \) is dense in \( X_0 \). Let \( p \in (1, \infty) \), and let \( A \in \mathcal{L}(X_1, X_0) \) be a closed (possibly unbounded) linear operator.

Let \( t_0,\;T > 0 \) and consider the interval \( J=[t_0,t_0+T] \). 
Let \( f: J \to X_0 \) be a given function. We consider the initial value problem:
\begin{equation}\label{fundamental_equation}
\begin{cases}
\displaystyle
\frac{du}{dt} + Au = f(t), & t_0<t<t_0+T, \\[8pt]
u(t_0) = 0.
\end{cases}
\end{equation}

We recall the following definition from \cite[Definition 1.1.]{Bari}.
\begin{definition}[Maximal \( L^p \)-regularity]\label{maximal_L^p_regularity}
The operator \( A \) is said to have the property of \emph{maximal \( L^p \)-regularity on the interval \( J \)} if and only if  for every \( f \in L^p(J; X_0) \) there exists a unique \( u \in W^{1,p}(J; X_0) \cap L^p(J; X_1) \)  satisfying \eqref{fundamental_equation} in the \( L^p \)-sense.
\end{definition}

By the closed graph theorem, it follows that there exists a  constant \( C > 0 \) such that the solution \( u \) to \eqref{fundamental_equation} satisfies the estimate
\begin{equation}\label{Maximal_inequality}
   \|u\|_{W^{1,p}(J; X_0) \cap L^{p}(J; X_1)} \leq C \|f\|_{L^{p}(J; X_0)}.
\end{equation}
Throughout the paper, we shall refer to the infimum of all such \(C\)'s satifying \eqref{Maximal_inequality} as \emph{the maximal regularity constant of \(A
\) on $J$}, denoted by $C(A;J)$.
It is clear that  $T\mapsto C(A;[t_0,t_0+T])$ is a non-decreasing function. 
Let \( I \subseteq \mathbb{R} \) be an interval. We define the function spaces
\begin{equation}\label{def_E1_and_E0_space}
    E_1(I) := W^{1,p}(I; X_0) \cap L^p(I; X_1), \quad \text{and} \quad E_0(I) := L^p(I; X_0),
\end{equation}
equipped with the norms
\begin{equation}\label{def_E1_and_E0}
    \| \cdot \|_{E_1(I)} := \| \cdot \|_{W^{1,p}(I; X_0) \cap L^p(I; X_1)}, \quad \text{and} \quad \| \cdot \|_{E_0(I)} := \| \cdot \|_{L^p(I; X_0)}.
\end{equation}

Let $X_p := (X_0, X_1)_{1-\frac{1}{p},p}$ be the (real) interpolation space. One can show that the norm on $X_p$ can be expressed as 
\begin{equation}\label{def_X_p}
\norm{f}_{X_p} := \inf \left\{ \norm{F}_{E_1([0,\infty))} \colon f = F(0) \right\},    
\end{equation}
see e.g. Chapter $1$ in \cite{interpolation_theory}. 
We also recall that, for every $u \in X_0 \cap X_1$, the interpolation inequality
\begin{align}\label{eq:interpolation-X_alpha}
    \|u\|_{X_p} \leq C_p \|u\|_{X_0}^{\frac{1}{p}} \|u\|_{X_1}^{1-\frac{1}{p}}
\end{align}
holds, where the constant $C_p > 0$ depends only on $p$, see e.g. Corollary $1.2.7$ in \cite{Analytic-semigroups-Lunardi}. 

Now, for each $n\in\mathbb{N} \cup \{\infty\}$, we consider a quasilinear problem
\begin{equation}\label{main_equation}
\begin{cases}
\displaystyle
\frac{d\phi_{n}}{dt} + A_n(t,\phi_{n})\phi_{n}= f_{n}(t, \phi_{n}), \qquad & t>t_{0}, \\[10pt] 
\phi_{n}(t_{0})= u^{n}_{0},  
\end{cases}
\end{equation}
where $u_0^n\in X_p$ and we impose the following assumptions on \( A_n \) and \( f_n \): 
\begin{enumerate}[label=\rm{(I.\alph*)}, ref=\rm{(I.\alph*)}]
\item \label{assumption_structure}
(\textit{Continuity assumptions}).  
The mapping
\[
A_n : [t_0,\infty) \times X_p \to \mathcal{L}(X_1,X_0)
\]
is continuous, and the function
\[
f_n : [t_0,\infty) \times X_p \to X_0
\]
is of \emph{Carath\'eodory type}, that is:
\begin{enumerate}[
    label=\rm{(I.\alph{enumi}.\roman*)},
    ref=\rm{(I.\alph{enumi}.\roman*)}
]
    \item \label{assumption_structure1}
    for each fixed \( u \in X_p \), the mapping
    \[
    t \longmapsto f_n(t,u)
    \]
    is measurable;

    \item \label{assumption_structure2}
    for almost every \( t \ge t_0 \), the mapping
    \[
    u \longmapsto f_n(t,u)
    \]
    is continuous.
\end{enumerate}

\item \label{existence_reason1}
(\textit{Lipschitz continuity of \( A_n \) in the second variable}).  
For every \( R>0 \), there exists a constant \( L_R^n>0 \) such that for all
\( u_1,u_2 \in X_p \), \( v \in X_1 \), and
\( t \in [t_0,t_0+R] \) with
\(
\|u_1\|_{X_p} \le R\),
 \(\|u_2\|_{X_p} \le R,
\)
the estimate
\[
\| A_n(t,u_1)v - A_n(t,u_2)v \|_{X_0}
\le
L_R^n \, \|u_1-u_2\|_{X_p}\, \|v\|_{X_1}
\]
holds.

\item \label{existence_reason2}
(\textit{Integrability of \( f_n(\cdot,0) \)}).  
The function
\(
f_n(\cdot,0)\) belongs to \(
L^p_{\mathrm{loc}}(t_0,\infty;X_0).
\)

\item \label{existence_reason3}
(\textit{Lipschitz continuity of \( f_n \) in the second variable}).  
For every \( R>0 \), there exists a function
\(
\Psi_R^n \in L^p_{\mathrm{loc}}([t_0,\infty))
\)
such that for all \( u_1,u_2 \in X_p \) with
\(
\|u_1\|_{X_p} \le R\),
 \(\|u_2\|_{X_p} \le R,
\)
the inequality
\[
\| f_n(t,u_1) - f_n(t,u_2) \|_{X_0}
\le
\Psi_R^n(t)\, \|u_1-u_2\|_{X_p}
\]
holds for almost every \( t \ge t_0 \).

\item \label{existence_reason4}
(\textit{Maximal \( L^p \)-regularity of \( A_n \)}).  
For each \( t \in [t_0,\infty) \) and each \( v \in X_p \),
the operator \( A_n(t,v) \) enjoys maximal
\( L^p \)-regularity on any compact interval.
\end{enumerate}

We note that assumptions \ref{assumption_structure}--\ref{existence_reason4} guarantee local existence and uniqueness of the solution to \eqref{main_equation} for all $n\in\N\cup\{\infty\}$ (see~\cite[Theorem 3.1]{Bari}).

Our goal is to show that for any compact interval $I\subset[t_0,\infty)$, we have $\phi_n\to\phi_\infty$ in $E_1(I)$ as $n\to \infty$. To this end, we will make further assumption on the convergence of $A_n, f_n, u_0^n$ to $A_\infty, f_\infty, u_0^\infty$  respectively as $n\to\infty$.

\section{Main abstract theorem on continuous dependence for strong solutions}\label{sec:cont_dep}

Before proving our main theorem, we first establish a lemma that plays a crucial role in the analysis. Although the maximal regularity constant associated with \( A(t, u) \) may, in general, depend on both \( t \) and \( u \), the following lemma guarantees that if \( (t, u) \) lies in a compact subset of \( [t_0, \infty) \times X_p \), then these constants remain uniformly bounded. As a result, a single uniform constant will suffice for all operators \( A(t, u) \) within this set.

\begin{lemma}\label{Introductory_lemma}
Let  \( A_\infty: [t_0, \infty) \times X_p \rightarrow \mathcal{L}(X_1, X_0) \) be a continuous mapping that satisfies assumptions \ref{existence_reason1} and \ref{existence_reason4}.
Fix $R>0$. Let $t\in [t_0, t_0+R]$, and \( u \in X_p \), and denote by $c_{\mathrm{M}}(t,u):=C(A_{\infty}(t,u);[t_0, t_0+R])$ the maximal regularity constant of $A_\infty(t,u)$ on $[t_0,t_0+R]$.

Then the mapping 
\[
[t_0, t_0+R] \times X_p \ni (t,u) \mapsto c_{\mathrm{M}}(t,u)\in \R,
\]
is continuous.
\end{lemma}

\begin{proof}
    Let us take $u_1$, $u_2 \in X_p$ and $t_1, t_2 \in [t_0, t_0+R]$. By \ref{existence_reason4}, we know that the operators \( A_\infty(t_1, u_1) \) and $A_\infty(t_2, u_2)$  have the property of maximal \( L^p \)-regularity on $[t_0, t_0+R]$. Therefore, for all $f\in E_0([t_0, t_0+R])$, there exists $v_1$, $v_2 \in E_1([t_0, t_0+R])$
    satisfying the equations
\begin{equation}\label{lemma_equation1}
\begin{cases}
\displaystyle
\frac{dv_1}{dt} + A_\infty(t_1,u_1)v_1= f(t), \qquad & t_0<t<t_{0}+R, \\[10pt]  
v_1(t_{0})= 0,  
\end{cases}
\end{equation}
\begin{equation}\label{lemma_equation2}
\begin{cases}
\displaystyle
\frac{dv_2}{dt} + A_\infty(t_2,u_2)v_2= f(t), \qquad & t_0<t<t_{0}+R, \\[10pt]  
v_2(t_{0})= 0,
\end{cases}
\end{equation}
in the \( L^p \)-sense, and
the inequalities
\begin{equation}\label{important_inequality1}
   \|v_1\|_{E_1([t_0, t_0+R])} \leq c_{\mathrm{M}}(t_1,u_1) \|f\|_{E_0([t_0, t_0+R])},
\end{equation}
\begin{equation}\label{important_inequality2}
   \|v_2\|_{E_1([t_0, t_0+R])} \leq c_{\mathrm{M}}(t_2,u_2) \|f\|_{E_0([t_0, t_0+R])}.
\end{equation}
We subtract equations~\eqref{lemma_equation1} and~\eqref{lemma_equation2} side by side and write 
\begin{equation}\label{lemma_equation3}
\begin{cases}
\displaystyle
\frac{d}{dt}(v_1-v_2) + A_\infty(t_1,u_1)(v_1-v_2)= F(t), \qquad & t_0<t<t_{0}+R, \\[10pt]  
(v_1-v_2)(t_{0})= 0,  
\end{cases}
\end{equation}
where $F(t):= A_\infty(t_2,u_2)v_2- A_\infty(t_1,u_1)v_2$. 

 First, we show that $F\in E_0([t_0, t_0+R])$. By using the triangle inequality and \ref{existence_reason1}, we get the following inequality:
\begin{equation}\label{estimate_F}
\begin{aligned}
    \norm{F}_{E_0([t_0, t_0+R])} &\leq \norm{A_\infty(t_2,u_2)v_2-A_\infty(t_1,u_2)v_2}_{E_0([t_0, t_0+R])}\\&\qquad+ \norm{A_\infty(t_1,u_2)v_2-A_\infty(t_1,u_1)v_2}_{E_0([t_0, t_0+R])} \\
    &\leq \norm{A_\infty(t_2,u_2)-A_\infty(t_1,u_2)}_{\mathcal{L}(X_1, X_0)}\norm{v_2}_{E_1([t_0, t_0+R])}\\&\qquad+
    L^\infty_R \norm{u_1-u_2}_{X_p}\norm{v_2}_{E_1([t_0, t_0+R])}.
\end{aligned}
\end{equation}
Now, combining \eqref{important_inequality2} and \eqref{estimate_F}, we get 
\begin{equation}
\begin{aligned}
     \norm{F}_{E_0([t_0, t_0+R])} &\leq \norm{A_\infty(t_2,u_2)-A_\infty(t_1,u_2)}_{\mathcal{L}(X_1, X_0)}c_{\mathrm{M}}(t_2,u_2)\norm{f}_{E_0([t_0, t_0+R])}\\&\qquad+
    L^\infty_R \norm{u_1-u_2}_{X_p}c_{\mathrm{M}}(t_2,u_2)\norm{f}_{E_0([t_0, t_0+R])}.
\end{aligned}
\end{equation}
Therefore, we have $F\in E_0([t_0, t_0+R])$. Since the operator \( A_\infty(t_1, u_1) \) has the property of maximal \( L^p \)-regularity on $[t_0, t_0+R]$, by taking $f= F$ and using \eqref{important_inequality1}, we get that 
\begin{equation}\label{v_1-v_2_estimate}
    \|v_1-v_2\|_{E_1([t_0, t_0+R])} \leq c_{\mathrm{M}}(t_1,u_1) \|F\|_{E_0([t_0, t_0+R])}.
\end{equation}
For $0<\varepsilon<1$ we observe  that, if we have the bounds $\norm{u_1-u_2}_{X_p} \leq \frac{\varepsilon}{2L^\infty_R c_{\mathrm{M}}(t_1,u_1)}$ and $\norm{A_\infty(t_2,u_2)-A_\infty(t_1,u_2)}_{\mathcal{L}(X_1, X_0)}\leq \frac{\varepsilon}{2 c_{\mathrm{M}}(t_1,u_1)}$, then by combining  \eqref{estimate_F} and \eqref{v_1-v_2_estimate}, we obtain 
\begin{equation}\label{final_v_1-v_2}
\begin{aligned}
    \|v_1-v_2\|_{E_1([t_0, t_0+R])} \leq \varepsilon\norm{v_2}_{E_1([t_0, t_0+R])}.
\end{aligned}
\end{equation}
Now, using the triangle inequality on \eqref{final_v_1-v_2}, rearranging, and using \eqref{important_inequality1}, we get
\begin{equation}
\begin{aligned}
   \left(1-\varepsilon\right)\|v_2\|_{E_1([t_0, t_0+R])} &\leq \|v_1\|_{E_1([t_0, t_0+R])} \\&\leq c_{\mathrm{M}}(t_1,u_1) \|f\|_{E_0([t_0, t_0+R])}.
   \end{aligned}
\end{equation}
Therefore, we have $c_{\mathrm{M}}(t_2,u_2)\leq\frac{c_{\mathrm{M}}(t_1,u_1)}{1-\varepsilon}$. Similarly, the roles of $(t_1, u_1)$ and $(t_2, u_2)$ can be interchanged to obtain also that $c_{\mathrm{M}}(t_1,u_1)\leq\frac{c_{\mathrm{M}}(t_2,u_2)}{1-\varepsilon}$ and to conclude that $(1-\varepsilon)c_{\mathrm{M}}(t_1,u_1)\leq c_{\mathrm{M}}(t_2,u_2)\leq\frac{c_{\mathrm{M}}(t_1,u_1)}{1-\varepsilon}.$

Using our previous analysis and knowing that \( A_\infty: [t_0, \infty) \times X_p \rightarrow \mathcal{L}(X_1, X_0) \) is a continuous mapping, it is immediate that the mapping 
\[
[t_0, t_0+R] \times X_p \ni (t,u) \mapsto c_{\mathrm{M}}(t,u)\in \R
\]
is continuous.
\end{proof}
\begin{remark}\label{eq:def-maximal-regularity-estimate}
    If \( U \subset X_p \) is compact, by Lemma \ref{Introductory_lemma}, it follows that the mapping
\[
[t_0, t_0+R] \times U \ni (t,u) \mapsto c_{\mathrm{M}}(t,u)\in \R
\]
is bounded. We will use this fact in the proof of next theorem. In fact, for a suitable choice of $R>0$, we will consider $U=\phi_{\infty}([t_0,t_0+R])$, where $\phi_\infty:[t_0,t_0+R]\to X_p$ is continuous. We will then fix a constant \( c_{\mathrm{M}} > 0 \) of the form $$c_{\mathrm{M}}:=\sup_{(t,u)\in [t_0,t_0+R]\times U} c_{\mathrm{M}}(t,u).$$ 
Lemma \ref{Introductory_lemma} ensures that above supremum is finite.
\end{remark}

We are now in the position to state and prove the main result of our paper.
\begin{theorem}\label{main theorem1}
  Suppose that for each $n\in\N\cup\{\infty\}$ the pair \( (A_n, f_n) \) satisfies  the assumptions ~\ref{assumption_structure}, \ref{existence_reason1}, \ref{existence_reason2}, \ref{existence_reason3}, and~\ref{existence_reason4}. 

Additionally, assume that the sequences \( (f_n)_{n \geq 1} \) and \( (A_n)_{n \geq 1} \) converge to \( f_\infty \) in \( X_0 \) and \( A_\infty \) in \( \mathcal{L}(X_1, X_0) \), respectively, uniformly on bounded subsets of \( [t_0, \infty) \times X_p \) as $n\to\infty$. Specifically, the following properties are satisfied:
\begin{enumerate}[label=\rm{(II)},ref=\rm{(II)}]
    \item \label{uniform_convergence}
    For every \( \varepsilon > 0 \), every compact subinterval \( K \subset [t_0, \infty) \), and every \( \eta > 0 \), there exists \( N \geq 1 \) such that for all \( n \geq N \), the following hold:
    
    \begin{enumerate}[label=\rm{(II.\alph*)},ref=\rm{(II.\alph*)}]
        \item \label{uniform_fn}
        \[
        \sup_{t \in K} \sup_{\|u\|_{X_p} \leq \eta} \|f_n(t, u) - f_\infty(t, u)\|_{X_0} < \varepsilon,
        \]
        
        \item \label{uniform_An}
        \[
        \sup_{t \in K} \sup_{\|u\|_{X_p} \leq \eta} \|A_n(t, u) - A_\infty(t, u)\|_{\mathcal{L}(X_1, X_0)} < \varepsilon.
        \]
    \end{enumerate}
\end{enumerate}

Let \( (u_0^n)_{n \geq 1} \subset X_p \) be a sequence of initial data satisfying
\begin{enumerate}[label=\rm{(III)},ref=\rm{(III)}]
    \item \label{covergenceof intialdata}
    \( \|u_0^n - u_0^\infty\|_{X_p} \to 0 \quad \text{as } n \to \infty \).
\end{enumerate}

For each \( n \geq 0 \), let \( \phi_n \in E_1(I)\) denote the unique solution to the problem \eqref{main_equation}, defined on the maximal interval of existence \( [t_0, t_0 + T_n) \), where \( I \) is any compact subinterval of \( [t_0, t_0 + T_n) \). Then:
\begin{itemize}
    \item The sequence \( (\phi_n)_{n \geq 1} \) converges to \( \phi_\infty \) in the following sense: for every compact subinterval \( K \subset [t_0, t_0 + T_\infty) \),
    \[
    \lim_{n \to \infty} \|\phi_n - \phi_\infty\|_{E_1(K)} = 0.
    \]
    \item The maximal existence time of the limit $\phi_\infty$ satisfies the following inequality:
    \begin{equation}\label{T_infty_limsup}
        T_\infty \leq \limsup_{n \to \infty} T_n.
    \end{equation}    
\end{itemize}
\end{theorem}
\begin{proof}
    Under assumptions~\ref{assumption_structure}, \ref{existence_reason1}, \ref{existence_reason2}, \ref{existence_reason3}, and~\ref{existence_reason4}, it follows that \( T_n > 0 \) for each \( n \in \mathbb{N} \cup \{\infty\} \) (see~\cite[Theorem 3.1]{Bari}).

 Let \( T_\infty' < T_\infty \) be fixed. Since \( \phi_\infty \in  C([t_0, t_0 + T_\infty']; X_p) \), there exists \( R > 0 \) such that
\[
\| \phi_\infty(t) \|_{X_p} \leq R, \quad \text{for all } t \in [t_0, t_0 + T_\infty'].
\]
Moreover, for every \( \epsilon > 0 \), there exists \( 0 < \delta_1 \leq 1 \) such that the following hold:
\begin{equation} \label{eq:equi-local-phi0}
\begin{aligned}
\| \phi_\infty(t_2) - \phi_\infty(t_1) \|_{X_p} &\leq \epsilon, 
\quad \text{for all } t_1, t_2 \in [t_0, t_0 + T_\infty'], \ \text{with } 0 \leq t_2 - t_1 \leq \delta_1.
\end{aligned}
\end{equation}
We also know that $\phi_\infty \in E_1([t_0, t_0 + T_\infty'])$. Therefore, for every \( \epsilon > 0 \), there exists \( 0 < \delta_2 \leq 1 \) such that the following hold:
\begin{equation} \label{eq:equi-local-phi_infty}
\begin{aligned}
\| \phi_\infty \|_{E_1([t_0, t_0 + \delta_1])} &\leq \epsilon.
\end{aligned}
\end{equation}
It is also given that \( A_\infty: [t_0, \infty) \times X_p \rightarrow \mathcal{L}(X_1, X_0) \) is continuous. Therefore, the map 
\[
(t, s) \mapsto A_\infty(t, \phi_{\infty}(s))
\]
is uniformly continuous on \( [t_0, t_0 + T_\infty'] \times [t_0, t_0 + T_\infty'] \); that is, for every \( \epsilon > 0 \), there exists \( 0 < \delta_3 \leq 1 \) such that
\begin{equation}
    \| A_\infty(t_1, \phi_0(s_1)) - A_\infty(t_2, \phi_0(s_2)) \|_{\mathcal{L}(X_1, X_0)} \leq \epsilon
\end{equation}
for all \( t_1, t_2, s_1, s_2 \in [t_0, t_0 + T_\infty'] \) with
\(
0 \leq t_2 - t_1 \leq \delta_2, \text{ and } 0 \leq s_2 - s_1 \leq \delta_2.
\)
Since, we know $\Psi^\infty_{2R} \in L^p(t_0, t_0 + T'_\infty)$. Then for every \( \epsilon > 0 \), there exists a constant $0 < \delta_4 \leq {T'_\infty}$ such that for all $a, b \in (t_0, t_0 + T'_\infty)$ with $|a - b| < \delta$, one has
\begin{align}\label{eq:Psi02R-epsilon}
    \|\Psi^\infty_{2R}\|_{L^p(a,b)} \leq \epsilon.
\end{align} 

Fix $\epsilon>0$, and let us take $\delta:= \min\{\delta_1, \delta_2, \delta_3, \delta_4,T_\infty'\}$. Define
\begin{equation*}
\begin{aligned}
        S^0_n :=   \{t\in (0, T_n) \colon t \leq \delta, \mbox{ }
        \norm{\phi_\infty(\tau)-\phi_n(\tau)}_{X_p} \leq \epsilon \mbox{ for all }\tau \in [t_0,t_0+t]\}.\label{def-t_n}   
\end{aligned}
\end{equation*}
From assumption~\ref{covergenceof intialdata}, together with the continuity of \( \phi_\infty \) and \( \phi_n \), it follows that there exists \( N_1 \in \mathbb{N} \) such that for all \( n \geq N_1 \), the set \( S^0_n \) is non-empty. Consequently, we define \(t_n := \sup S^0_n > 0
\) for all \( n \geq N_1 \).
Our goal is to show that \( t_n = \delta \) for large $n\in\N$.

Fix \( n \geq N_1 \). We begin by rewriting the problem~\eqref{main_equation} in the following form:
\begin{equation}\label{main_equation1}
\begin{cases}
\displaystyle
\frac{d\phi_{n}}{dt} + A_\infty(t_0, u_0^\infty)\phi_n 
= f_n(t, \phi_n) - A_n(t, \phi_n)\phi_n + A_\infty(t_0, u_0^\infty)\phi_n, 
& \quad t > t_0, \\[10pt]
\phi_n(t_0) = u_0^n.
\end{cases}
\end{equation}
Similarly, for ``\( n = \infty \)'', we consider:
\begin{equation}\label{main_equation2}
\begin{cases}
\displaystyle
\frac{d\phi_{\infty}}{dt} + A_\infty(t_0, u_0^\infty)\phi_\infty 
= f_\infty(t, \phi_\infty) - A_\infty(t, \phi_\infty)\phi_\infty + A_\infty(t_0, u_0^\infty)\phi_\infty, 
& \quad t > t_0, \\[10pt]
\phi_\infty(t_0) = u_0^\infty.
\end{cases}
\end{equation}
 By subtracting equations~\eqref{main_equation1} and~\eqref{main_equation2} side by side, and using the fact that \( A_\infty(t_0, u^\infty_0) \)  has the property of maximal \( L^p \)-regularity on the interval \( [t_0, t_0 + \delta] \), we deduce the following estimate for all \( t_n' \leq t_n \leq T'_\infty \):
\begin{equation}\label{Maximal_regularity1}
\begin{aligned}
   &\norm{ \phi_{n} - \phi_{\infty} }_{E_1([t_0, t_0 + t'_n])}\leq c_{\mathrm{M}}\left( \norm{ u^{n}_{0} - u^{\infty}_{0}}_{X_{p}} + \norm{f_{n}(\cdot, \phi_{n}) - f_\infty(\cdot, \phi_{\infty})}_{E_0([t_0, t_0 + t'_n])}\right)\\ 
   & +c_{\mathrm{M}}\norm{A_\infty(t_0, u_0^\infty)\phi_\infty - A_\infty(\cdot, \phi_\infty)\phi_\infty - A_\infty(t_0, u_0^\infty)\phi_n + A_n(\cdot, \phi_n)\phi_n}_{E_0([t_0, t_0 + t'_n])},
\end{aligned}
\end{equation}
where \( c_{\mathrm{M}} \) denotes the maximal regularity constant defined in Remark~\ref{eq:def-maximal-regularity-estimate}. By using the triangle inequality on \eqref{Maximal_regularity1}, we get
\begin{equation}\label{Maximal_regularity}
    \begin{aligned}
&\norm{ \phi_{n} - \phi_{\infty} }_{E_1([t_0, t_0 + t'_n])}\\& \qquad\leq c_{\mathrm{M}}\left( \norm{ u^{n}_{0} - u^{\infty}_{0}}_{X_{p}} + \norm{ f_{n}(\cdot, \phi_{n}) - f_\infty(\cdot, \phi_{\infty})}_{E_0([t_0, t_0 + t'_n])}\right)\\ 
&\qquad\quad + c_{\mathrm{M}}\norm{(A_\infty(\cdot, u_0^\infty) - A_\infty(\cdot, \phi_\infty))(\phi_\infty - \phi_n)}_{E_0([t_0, t_0 + t'_n])} \\
&\qquad\quad + c_{\mathrm{M}}\norm{(A_\infty(\cdot, \phi_\infty) - A_\infty(\cdot, \phi_n))\phi_n}_{E_0([t_0, t_0 + t'_n])} \\
&\qquad\quad + c_{\mathrm{M}}\norm{A_n(\cdot, \phi_n)\phi_n - A_\infty(\cdot, \phi_n)\phi_n}_{E_0([t_0, t_0 + t'_n])}\\
&\qquad\quad + c_{\mathrm{M}}\norm{(A_\infty(t_0, u_0^\infty) - A_\infty(\cdot, u_0^\infty))(\phi_\infty - \phi_n)}_{E_0([t_0, t_0 + t'_n])}\\
&\qquad =: c_{\mathrm{M}} \sum_{i=1}^6 T_i,
    \end{aligned}
\end{equation}

Let us estimate the term $T_2$ in \eqref{Maximal_regularity}. By using the triangle inequality and assumption \ref{existence_reason3}, we get 
\begin{equation}\label{f_estimate1}
\begin{aligned}
T_2 &:= \,
\norm{ f_{n}(\cdot, \phi_{n}) - f_\infty(\cdot, \phi_{\infty})}_{E_0([t_0, t_0 + t'_n])}\\
&\leq \Delta_{n} \, \delta^{1/p} + \norm{ \Psi^\infty_{2R}(\phi_{n} - \phi_{\infty})}_{L^{p}(t_0, t_0 + t'_n; X_p)} \\
&\leq \Delta_{n} \, \delta^{1/p} + \|\Psi^\infty_{2R}\|_{L^p(t_0, t_0 + \delta)} \sup_{t \in [t_0, t_0 + t'_n]} \norm{ \phi_{n}(t) - \phi_{\infty}(t) }_{X_p},
\end{aligned}
\end{equation}

where $\Delta_{n} := \sup \{\norm{ f_{n}(s, x) - f_{0}(s, x)} _{X_0} \colon t_0\leq s \leq t_0+\delta, \norm{x}_{X_p} \leq 2R  \}$.
We know that the following inequality holds (see Theorem~\ref{Embedding theorem} in Appendix~\ref{sec:Embedding theorem}): for all $t_n' > 0$
\begin{align}\label{embedding}
    \sup_{t \in [t_0, t_0+t_n']} \norm{ \phi_{n}(t) - \phi_{\infty}(t) }_{X_p} \leq c_1\left(\norm{ u^{n}_{0}- u^{0}_{0}}_{X_{p}}+\norm{ \phi_{n} - \phi_{\infty} }_{E_1([t_0, t_0 + t'_n])}\right),
\end{align}
where $c_1$ is a positive constant independent of $t_n'$. 
Therefore, by combining \eqref{f_estimate1} and \eqref{embedding}, and using \eqref{eq:Psi02R-epsilon}, we get the following estimate:
\begin{equation}\label{final_2}
\begin{aligned}
    T_2&\leq\Delta_{n} {\delta}^\frac{1}{p}+
    \|\Psi^\infty_{2R}\|_{L^p(t_0,t_0+\delta)}c_1\left(\norm{ u^{n}_{0}- u^{\infty}_{0}}_{X_{p}}+\norm{ \phi_{n} - \phi_{\infty} }_{E_1([t_0, t_0 + t'_n])}\right)\\
    &\leq \Delta_{n} {\delta}^\frac{1}{p}+ \epsilon c_1 \left(\norm{ u^{n}_{0}- u^{\infty}_{0}}_{X_{p}}+\norm{ \phi_{n} - \phi_{\infty} }_{E_1([t_0, t_0 + t'_n])}\right).
\end{aligned}
\end{equation}
We estimate the term $T_3$ in \eqref{Maximal_regularity}.
From \ref{existence_reason1}  and \eqref{eq:equi-local-phi0}, we get that 
\begin{equation}\label{final_3}
    \begin{aligned}
T_3 &:= \,\norm{(A_\infty(\cdot,u_0^\infty) - A_\infty(\cdot, \phi_\infty(\cdot)))(\phi_\infty(\cdot) - \phi_n(\cdot))}_{E_0([t_0, t_0 + t_n'])} \\
&\leq L^\infty_{2R} \sup_{t \in [t_0, t_0 + \delta]} \norm{u_0^\infty - \phi_\infty(t)}_{X_{p}} \norm{\phi_n - \phi_\infty}_{E_1([t_0, t_0 + t_n'])} \\
&\leq L^\infty_{2R} \epsilon \norm{ \phi_n - \phi_\infty }_{E_1([t_0, t_0 + t_n'])}.
    \end{aligned}
\end{equation}
 Concerning the term $T_4$ in \eqref{Maximal_regularity}, from \ref{existence_reason1} and \eqref{eq:equi-local-phi0}, we get that 
\begin{equation}\label{fourth_estimate}
    \begin{aligned}
        T_4 &:= \,\norm{(A_\infty(\cdot,\phi_\infty)- A_\infty(\cdot,\phi_n))\phi_n}_{E_0([t_0, t_0 + t'_n])}\\
        &\leq L^\infty_{2R}\sup_{t \in [t_0, t_0 + t'_n]}\norm{ \phi_\infty(t) - \phi_n(t)}_{X_{p}} \norm{\phi_n}_{E_1([t_0, t_0 + t'_n])}\\
        &\leq L^\infty_{2R}\sup_{t \in [t_0, t_0 + t'_n]}\norm{ \phi_\infty(t) - \phi_n(t)}_{X_{p}} \norm{\phi_n - \phi_\infty}_{E_1([t_0, t_0 + t'_n])} \\
        &\qquad+ L^\infty_{2R}\sup_{t \in [t_0, t_0 + t'_n]}\norm{ \phi_\infty(t) - \phi_n(t)}_{X_{p}} \norm{\phi_\infty}_{E_1([t_0, t_0 + t'_n])}\\
        &\leq L^\infty_{2R} \epsilon \norm{\phi_n - \phi_\infty}_{E_1([t_0, t_0 + t'_n])} + L^\infty_{2R} \sup_{t \in [t_0, t_0 + t'_n]}\norm{ \phi_\infty(t) - \phi_n(t)}_{X_{p}} \epsilon.
    \end{aligned}
\end{equation}
Therefore, by combining \eqref{fourth_estimate} and \eqref{embedding} we get the following the estimate:
\begin{equation}\label{final_4}
    \begin{aligned}
        T_4 &\leq \,
         L^\infty_{2R}\epsilon \left(\norm{\phi_n - \phi_\infty}_{E_1([t_0, t_0 + t'_n])}
         + c_1\left(\norm{ u^{n}_{0} - u^{\infty}_{0}}_{X_{p}} + \norm{ \phi_{n} - \phi_{\infty} }_{E_1([t_0, t_0 + t'_n])}\right) \right).
    \end{aligned}
\end{equation}
We estimate the term $T_5$ in \eqref{Maximal_regularity} using \ref{eq:equi-local-phi_infty} as follows:
\begin{equation}\label{final_5}
    \begin{aligned}
        T_5 &:= \,\norm{A_n(\cdot, \phi_n)\phi_n - A_\infty(\cdot, \phi_n)\phi_n}_{E_0([t_0, t_0 + t'_n])} \\
        &\leq \tilde{\Delta}_{n} \norm{\phi_n}_{E_1([t_0, t_0 + t'_n])} \\
        &\leq \tilde{\Delta}_{n} \norm{\phi_n - \phi_\infty}_{E_1([t_0, t_0 + t'_n])} + \tilde{\Delta}_{n} \norm{\phi_\infty}_{E_1([t_0, t_0 + t'_n])} \\
        &\leq \tilde{\Delta}_{n} \norm{\phi_n - \phi_\infty}_{E_1([t_0, t_0 + t'_n])} + \tilde{\Delta}_{n} \epsilon,
    \end{aligned}
\end{equation}
where $\tilde{\Delta}_{n} := \sup \{\norm{A_n(s, x)- A(s,x)}_{\mathcal{L}(X_1,X_0)} \colon t_0\leq s \leq t_0+\delta, \norm{x}_{X_p} \leq 2R  \}$. In the similar fashion, we find the following estimate for the term $T_6$ in \eqref{Maximal_regularity}.
\begin{equation}\label{final_6}
    \begin{aligned}
T_6 &:= \,\norm{(A_\infty(t_0,u_0^\infty) - A_\infty(\cdot, u_0^\infty))(\phi_\infty - \phi_n)}_{E_0([t_0, t_0 + t'_n])} \\
         &\leq \sup_{t \in [t_0, t_0 + t'_n]} \norm{A_\infty(t_0,u_0^\infty) - A_\infty(t, u_0^\infty)}_{\mathcal{L}(X_1, X_0)} \norm{\phi_n - \phi_\infty}_{E_1([t_0, t_0 + t'_n])} \\
         &\leq \epsilon \norm{\phi_n - \phi_\infty}_{E_1([t_0, t_0 + t'_n])}.
    \end{aligned}
\end{equation}

Combining \eqref{Maximal_regularity}, \eqref{final_2}, \eqref{final_3}, \eqref{final_4}, \eqref{final_5} and \eqref{final_6}, we get
\begin{equation}\label{combined}
    \begin{aligned}
      &\norm{ \phi_n - \phi_\infty }_{E_1([t_0, t_0 + t'_n])} \\
      &\qquad\leq c_{\mathrm{M}} \norm{ u^n_0 - u^\infty_0 }_{X_p} + c_{\mathrm{M}} \Delta_n \delta^{\frac{1}{p}} \\
      &\qquad\quad + c_{\mathrm{M}} \epsilon c_1 \left( \norm{ u^n_0 - u^\infty_0 }_{X_p} + \norm{ \phi_n - \phi_\infty }_{E_1([t_0, t_0 + t'_n])} \right) \\
      &\qquad\quad + c_{\mathrm{M}} L^\infty_{2R} \epsilon \norm{ \phi_n - \phi_\infty }_{E_1([t_0, t_0 + t'_n])} \\
      &\qquad\quad + c_{\mathrm{M}} L^\infty_{2R} \epsilon \norm{ \phi_n - \phi_\infty }_{E_1([t_0, t_0 + t'_n])} \\
      &\qquad\quad + c_{\mathrm{M}} L^\infty_{2R} c_1 \epsilon \left( \norm{ u^n_0 - u^\infty_0 }_{X_p} + \norm{ \phi_n - \phi_\infty }_{E_1([t_0, t_0 + t'_n])} \right) \\
      &\qquad\quad + c_{\mathrm{M}} \tilde{\Delta}_n \norm{ \phi_n - \phi_\infty }_{E_1([t_0, t_0 + t'_n])} + c_{\mathrm{M}} \tilde{\Delta}_n \epsilon \\
      &\qquad\quad + c_{\mathrm{M}} \epsilon \norm{ \phi_n - \phi_\infty }_{E_1([t_0, t_0 + t'_n])}.
    \end{aligned}
\end{equation}
Take $0<\epsilon \leq \frac{R}{4}$ such that $\epsilon\left(c_{\mathrm{M}}c_1+2c_{\mathrm{M}}L^0_{2R}+c_{\mathrm{M}}L^0_{2R}c_1+c_{\mathrm{M}}\right) \leq \frac12$. By \eqref{combined} we then have that 
\begin{equation}\label{simplified}
    \begin{aligned}
      \norm{ \phi_{n} - \phi_{\infty} }_{E_1([t_0, t_0 + t'_n])}&\leq 2c_{\mathrm{M}} \norm{ u^{n}_{0} - u^{\infty}_{0} }_{X_{p}}
      + 2c_{\mathrm{M}} \Delta_{n} \delta^{1/p}\\
      &\quad+ 2c_{\mathrm{M}} c_1 \norm{ u^{n}_{0} - u^{\infty}_{0} }_{X_{p}}
       + 2c_{\mathrm{M}} L^\infty_{2R} c_1 \norm{ u^{n}_{0} - u^{\infty}_{0} }_{X_{p}}\\
      &\quad+ 2c_{\mathrm{M}} \tilde{\Delta}_{n} \norm{ \phi_n - \phi_\infty }_{E_1([t_0, t_0 + t'_n])}
      + 2c_{\mathrm{M}} \tilde{\Delta}_{n},
    \end{aligned}
\end{equation}
 for all $t_n' \leq t_n (\leq T'_\infty)$. By \ref{uniform_An} we choose $n \in \N$ large enough such that $2c_{\mathrm{M}}\tilde{\Delta}_n \leq \frac{1}{2}$. Therefore, for all $t_n' \leq t_n (\leq T'_\infty)$ we have
\begin{equation}\label{final}
    \begin{aligned}
    \norm{ \phi_{n} - \phi_{\infty} }_{E_1([t_0, t_0 + t'_n])}
      &\leq 4c_{\mathrm{M}} \norm{ u^{n}_{0} - u^{\infty}_{0} }_{X_{p}}
      + 4c_{\mathrm{M}} \Delta_{n} \delta^{1/p}
      + 4c_{\mathrm{M}} c_1 \norm{ u^{n}_{0} - u^{\infty}_{0} }_{X_{p}}\\
      &\quad+ 4c_{\mathrm{M}} L^\infty_{2R} c_1 \norm{ u^{n}_{0} - u^{\infty}_{0} }_{X_{p}}
     + 4c_{\mathrm{M}} \tilde{\Delta}_{n}.
    \end{aligned}
\end{equation}
Therefore, by \ref{uniform_convergence}, \ref{covergenceof intialdata}, \eqref{embedding} and \eqref{final}, we conclude that for large $n \in \N$, we have
\begin{equation}\label{conclusion_1}
\sup_{t \in [t_0, t_0+\delta]} \norm{ \phi_{n}(t) - \phi_{\infty}(t) }_{X_p} \leq \frac{\epsilon}{2}.
\end{equation}
So, for large $n \in \N$, we get $t_n = \delta$. Now, if $\delta < T'_\infty$, then we define $S^1_n$ as follows:
\begin{equation*}
\begin{aligned}
        S^1_n :=   \{t\in (0, T_n) \colon t \leq \delta, \mbox{ } 
        \norm{\phi_\infty(\tau)-\phi_n(\tau)}_{X_p} \leq \epsilon \mbox{ for all }\tau \in [t_0+\delta,t_0+\delta+t]\}.\label{def-S^1_n}   
\end{aligned}
\end{equation*}
It is immediate from \eqref{conclusion_1} that $S^1_n$ is non-empty for large $n \in \N$.
The same analysis as before implies that:
\begin{equation}\label{conclusion_2}
\sup_{t \in [t_0+\delta, t_0+2\delta]} \norm{ \phi_{n}(t) - \phi_{\infty}(t) }_{X_p} \leq \frac{\epsilon}{2}.
\end{equation}
By induction, we conclude that for large $n \in \N$, we have
\begin{equation}\label{final_conclusion}
\sup_{t \in [t_0, t_0+T'_\infty]} \norm{ \phi_{n}(t) - \phi_{\infty}(t) }_{X_p} \leq \epsilon.
\end{equation}
Finally, we obtain that \eqref{final} holds for $t_n'=T'_\infty$ and we have that
\begin{equation}
    \lim_{n\to\infty}\norm{ \phi_{n}- \phi_{\infty}}_{E_1([t_0, t_0 + T'_\infty])}= 0.
\end{equation} 
The previous analysis also ensures that, for all \( k \in \mathbb{N} \), there exists \( n_k \in \mathbb{N} \) such that \( T_\infty - \frac{1}{k} \leq t_{n_k}\leq T_{n_k} \). The inequality \eqref{T_infty_limsup} then follows immediately.
\end{proof}
\section{Viscosity limits of non-Newtonian fluid models}\label{application}
In this section, we apply Theorem \ref{main theorem1} to prove our second main result concerning the continuous dependence upon data and viscosity coefficients of solutions to certain non-Newtonian, viscous and incompressible fluid models like those used in hemodynamics.  

Let  $d=2,3$ and $\Omega \subset \mathbb{R}^d$ be a bounded domain with boundary 
$\Gamma := \partial \Omega$ of class $C^{3-}$, representing the region occupied by the fluid. 
The boundary $\Gamma$ is decomposed into two mutually disjoint, relatively open and closed subsets,
\[
\Gamma = \Gamma_0 \cup \Gamma_s.
\]
The outward normal vector to $\Omega$ at a point $x \in \Gamma$ is denoted by $\nu = \nu(x)$.
For each $n\in \N\cup\{\infty\}$, we consider the following  initial-boundary value problem:
\begin{equation}\label{eq:motion}
\left\{
\begin{aligned}
   &\nabla\cdot v_n 
   =0\qquad&&\text{in }(0,T)\times\Omega,
   \\
   &\partial_t (\rho v_n)+ \rho (v_n\cdot \nabla)v_n 
   =\nabla\cdot \mathcal{S}_n(v_n)  \qquad&&\text{in } (0,T)\times \Omega,
   \\
   &\mathcal{S}_n(v_n,\pi_n)= 2\mu_n(\lvert D(v_n)\rvert^2)D(v_n)-\pi_n\mathbf{I},\\
   &\mathcal{B}_j(v_n)v_n=0\qquad&&\text{on }(0,T)\times \Gamma_j, \quad \text{ for }j\in\{0,s\},
   \\
   &v_n(0)=v^n_0\qquad&&\text{in } \Omega.
\end{aligned}
\right.    
\end{equation}
Here $v_n$ denotes the fluid velocity, $\pi_n$ its pressure (the Lagrange multiplier due to the incompressibility constraint), $\rho > 0$ is the fluid (constant) density, and 
$\mathcal{S}_n$ is the Cauchy stress tensor which satisfies the {\em constitutive equation} given in the third equation of \eqref{eq:motion}, with $\mathbf{I}$ the identity tensor and 
\[
D(v_n) := \frac{1}{2}\left(\nabla v_n + (\nabla v_n)^{T}\right)
\]
denoting the rate of deformation tensor  
with components
\[
D_{ij}(v_n) = \tfrac12\left(\partial_{x_i} v_{n,j} + \partial_{x_j} v_{n,i}\right)
\]
with respect to a Cartesian coordinate system. Additionally, we use the notation $|D(v_n)|^2=\sum_{i,j}^d (D_{ij}(v_n))^2$, and the non-constant viscosity coefficient $\mu_n$ is given by the map 
\begin{equation}\label{eq:viscosity}
[0,\infty)\ni s\mapsto \mu_n(s)=\mu_\infty+\eta_n(1+s)^{\frac{d-2}{2}}
\end{equation}
with $\mu_\infty>0$ constant and the sequence of non-negative real numbers $\{\eta_n\}_{n\in \N}$ such that $\eta_n\to 0$ as $n\to\infty$. We note that, in the particular case $\eta_n\equiv 0$, \eqref{eq:motion}$_3$ gives a linear dependence of the Cauchy stress tensor on the rate of deformation tensor characterizing the classical {\em Newtonian} viscous fluids. In such a case, \eqref{eq:motion} reduces to the classical {\em Navier-Stokes equations}. Thus, equations \eqref{eq:motion} is a {\em generalized Newtonian model} for non-Newtonian, viscous and incompressible fluids. 
This type of constitutive equation emerges in hemodynamics problems for the study of blood flows.

For $j\in\{0,s\}$, we denote by $\mathcal{B}_j(v_n)$ the boundary operators on $\Gamma_j$, 
which are given by
\[
\mathcal{B}_0(v_n) v_n= v_n \quad \text{on } \Gamma_0,
\]
and
\[
\mathcal{B}_s(v_n) v_n =
\left( v_n \cdot \nu,\;
 \mathcal{S}_n(v_n,\pi_n)\nu-(\nu\cdot \mathcal{S}_n(v_n,\pi_n)\nu)\nu
\right)
\quad \text{on } \Gamma_s.
\]
From a physical perspective, $\mathcal{B}_0(v)v=0$ corresponds to  a {\em no-slip boundary condition}, whereas $\mathcal{B}_s(v)v=0$ corresponds to a boundary condition of {\em pure slip}.

By expanding the term $\nabla \cdot \big(2\mu_n(\lvert D(v_n)\rvert^2)D(v_n)\big)$, the system
\eqref{eq:motion} can be rewritten in the following form:
\begin{equation}\label{eq:motion1}
\left\{
\begin{aligned}
   &\nabla \cdot v_n = 0
   && \text{in } (0,T)\times \Omega, \\
   &\partial_t (\rho v_n)
   + \mathcal{A}_n(v_n)\, v_n
   + \nabla \pi_n
   = F(v_n)
   && \text{in } (0,T)\times \Omega, \\
   &\mathcal{B}_j(v_n)v_n=0\qquad&&\text{on }(0,T)\times \Gamma_j, \quad j\in\{0, s\}, \\
   &v_n(0) = v_0^n
   && \text{in } \Omega.
\end{aligned}
\right.
\end{equation}

Here $\mathcal{A}_n(u)$ denotes the quasilinear differential operator 
\begin{equation}\label{def_script_A_n}
\begin{aligned}
\mathcal{A}_n(u)v
:= -\sum_{i=1}^d \Bigg\{ &
\mu_n\!\left(|D(u)|^2\right)\,\Delta v_i
+ \mu_n\!\left(|D(u)|^2\right)\,\partial_{x_i}(\nabla\!\cdot v) \\
&\quad
+ \sum_{k,j,\ell=1}^d
4\,\mu_n'\!\left(|D(u)|^2\right)\,
D_{ik}(u)\,D_{j\ell}(u)\,
\partial_{x_k} D_{j\ell}(v)
\Bigg\} \mathbf{e}_i ,
\end{aligned}
\end{equation}
where $\{\mathbf{e}_i\}_{i=1}^d$ is the canonical basis of $\mathbb{R}^d$ (corresponding to the given Cartesian coordinate system). Furthermore, the nonlinear 
term in \eqref{eq:motion1} is given by
\[
F(v) := -(\rho\,v \cdot \nabla) v .
\]
We aim to prove that ``sufficiently regular'' solutions to \eqref{eq:motion1} (resp. \eqref{eq:motion}) converge to the solution of the corresponding Navier-Stokes system in the limit $\eta_n\to 0$ and $v_0^n\to v^\infty_0$ as $n\to \infty$, in an appropriate topology. In order to achieve such a result, we plan to use Theorem \ref{main theorem1}. As a first step, we need to rewrite  \eqref{eq:motion1} in the form of \eqref{main_equation}. 

We start by  setting 
\[
X_0:=L_\sigma^{p}(\Omega)=\{v\in L^{p}(\Omega):\; \nabla\cdot  v=0\quad\text{in }\Omega,\ v\cdot \nu=0\quad\text{on }\Gamma\} \quad \text{for }p\in [1,\infty),
\]
where the divergence condition holds in the sense of distributions, while the boundary condition holds in the weak sense. 
We also introduce the Banach space 
\[
X_1 := \bigl\{
v \in H_p^2(\Omega)\cap X_0 \;:\;
v = 0 \ \text{on } \Gamma_0,\;
v \cdot \nu = 0 \ \text{and}\;
 D(v) \nu= (D(v)\cdot\nu) \nu  \ \text{ on } \Gamma_s
\bigr\},
\]
where $H_p^2(\Omega)$ denotes the Bessel potential space. 
We define the operators
\begin{equation}\label{def_A_n-f_n-application}
A_n(t,u) := \mathbb{P}\,\mathcal{A}_n(u),
\qquad
f_n(t,u) := \mathbb{P}\,F(u),
\end{equation}
where \(\mathbb{P}\) denotes the Helmholtz projection from \(L^p(\Omega)\) onto \(X_0\) (see \cite[Theorem~1.2]{Galdi}).

With these definitions, the system of equations \eqref{eq:motion1} can be rewritten as the evolution equation \eqref{main_equation} on the Banach space \(X_0\):
\begin{equation}\label{evolution-equation-application}
\begin{cases}
\displaystyle
\frac{d\phi_{n}}{dt} + A_n(t,\phi_{n})\,\phi_{n}= f_{n}(t, \phi_{n}),
& t>0, \\[8pt]
\phi_{n}(0)= v^{n}_{0},
\end{cases}
\end{equation}
where $\phi_n= v_n$ and $v^n_0\in X_p \subset W^{2-\frac{2}{p},p}(\Omega)$, where \(X_p\) is endowed with the norm inherited from 
\(W^{2-\frac{2}{p},p}(\Omega)\). Throughout our calculations, we use the following characterization of the fractional Sobolev space (more generally, Besov space) $W^{2s,p}(\Omega)=B_{pp}^{2s}(\Omega)=(X_0,X_1)_{s,p}$ as real interpolation of the spaces $X_0$ and $X_1$ defined above.

 We are now ready to state the main result of this section.
\begin{theorem}\label{th:main-Navier-Stokes}
    Let $p>d+2$ and the data in \eqref{eq:motion} satisfy the following conditions: 
    \begin{itemize}
        \item[(a)] The initial data $(v^n_0)_{n\in \N}$ and $v^\infty_0\in W^{2-\frac 2p,p}(\Omega)$ are such that $v^n_0\to v^\infty_0$ in $W^{2-\frac 2p,p}(\Omega)$ as $n\to \infty$.
        \item[(b)] The sequence of {\em generalized viscosity coefficients} $(\eta_n)_{n\in \N}$ in \eqref{eq:viscosity} converges to zero as $n\to \infty$.  
    \end{itemize} 
    Then, the sequence \( (v_n)_{n \in \N}\) of solutions to \eqref{evolution-equation-application} (resp. \eqref{eq:motion}) on the maximal time interval $[0,T_n)$ 
    converges to the (unique) solution $v_\infty$ of the following evolution equation for the Navier-Stokes system 
    \begin{equation}\label{eq:abstract-NS}
    \left\{\begin{aligned}
        &\frac{d v_\infty}{dt}=\mathbb P(\mu_\infty\Delta v_\infty-(\rho\, v_\infty\cdot\nabla)v_\infty) && \text{in } (0,T_\infty), 
        \\
        &v_\infty(0) = v_0^\infty,
    \end{aligned}\right.
    \end{equation}
    in the following sense: for every compact subinterval \( K \subset [0, T_\infty) \),
    \[
    \lim_{n \to \infty} \|v_n - v_\infty\|_{E_1(K)} = 0,
    \]
    where we recall that $E_1(K):=W^{1,p}(K,X_0)\cap L^p(K,X_1)$ (cf. \eqref{def_E1_and_E0_space} and \eqref{def_E1_and_E0}). 
    
    In addition, the maximal existence time of the limit $v_\infty$ satisfies the following inequality:
    \begin{equation}\label{T_infty_limsup1}
        T_\infty \leq \limsup_{n \to \infty} T_n.
    \end{equation}
\end{theorem}

As already mentioned, Theorem \ref{th:main-Navier-Stokes} will follow from Theorem \ref{main theorem1}. We dedicate next lemmas to verify  that assumptions~\ref{assumption_structure}-- \ref{existence_reason4} hold in our case. The restriction on the Lebesgue exponent $p$ will follow accordingly. We first note that, thanks to definition \eqref{def_A_n-f_n-application}, conditions \ref{assumption_structure1} and \ref{existence_reason3} hold trivially.

\begin{lemma}\label{Nonlinearity_of_f}
    Let $p>\frac{d+2}{2}$. For all $n\in \N\cup\{\infty\}$, $f_n(t,u)$ satisfies \ref{assumption_structure2} and  \ref{existence_reason3}.
\end{lemma}
\begin{proof}
By the Sobolev embedding theorem,
\[
W^{s,p}(\Omega)\hookrightarrow L^\infty(\overline{\Omega})
\quad \text{for } s>\frac{d}{p}.
\]
Moreover, for \( s>1 \), we have the continuous embedding
\[
W^{s,p}(\Omega)\hookrightarrow W^{1,p}(\Omega).
\]
Hence, for
\[
s>\max\left\{1,\frac{d}{p}\right\}
\]
and all \(u,v\in W^{s,p}(\Omega)\), we obtain
\begin{equation}\label{non-linear_estimate}
\begin{aligned}
\|(v\cdot\nabla)v-(u\cdot\nabla)u\|_{L^p(\Omega)}
&\le \|v\|_{L^\infty(\Omega)}\|v-u\|_{W^{1,p}(\Omega)}
     +\|v-u\|_{L^\infty(\Omega)}\|u\|_{W^{1,p}(\Omega)} \\
&\le C\Big(\|v\|_{W^{s,p}(\Omega)}+\|u\|_{W^{s,p}(\Omega)}\Big)
      \|v-u\|_{W^{s,p}(\Omega)},
\end{aligned}
\end{equation}
where \(C>0\)  is a constant depending only on \(\Omega\). Consequently, the nonlinear mapping
\[
v\mapsto (v\cdot\nabla)v\
\]
is well defined from \(W^{s,p}(\Omega)\) to \(L^p(\Omega)\), and bilinear whenever \( s > \max\left\{1, \frac{d}{p} \right\} \). Additionally, if \(u,v\in X_p\) with
\[
\|u\|_{X_p},\ \|v\|_{X_p}\le R
\]
for some $R>0$, then we get
\begin{equation}\label{eq:nonlinear-difference}
\|(v\cdot\nabla)v-(u\cdot\nabla)u\|_{X_0}
\le 2CR\,\|v-u\|_{X_p}
\end{equation}
whenever $p>\frac{d+2}{2}$. 

Finally, since the Helmholtz projection is bounded on \(L^p(\Omega)\), conditions \ref{existence_reason3} and \ref{assumption_structure2} immediately follow. 
\end{proof}
 
\begin{lemma}\label{nonlinearity_of_A}
    Let $p>d+2$. For all $n\in \N\cup\{\infty\}$, the operator $A_n(t,u)$ defined in \eqref{def_A_n-f_n-application} satisfies \ref{existence_reason1}.
\end{lemma}
\begin{proof}
Let $u_1,u_2 \in X_p$ satisfy
\begin{equation}\label{assumption_u1_u2<eta}
\|u_1\|_{X_p},\ \|u_2\|_{X_p} \le \eta
\end{equation}
with some $\eta>0$. The Sobolev embedding
\[
W^{2-\frac{2}{p},p}(\Omega)\hookrightarrow C^{1,\alpha}(\overline{\Omega}),
\qquad
\alpha=1-\frac{2}{p}-\frac{d}{p}>0,
\]
holds. In particular, we have the embedding
\begin{equation}\label{embedding_W1inf}
W^{2-\frac{2}{p},p}(\Omega)\hookrightarrow W^{1,\infty}(\Omega).
\end{equation}
Recalling the definition of the operators \eqref{def_A_n-f_n-application} and \eqref{def_script_A_n}, and the boundedness of the Helmholtz projection on $L^p(\Omega)$, it is enough to find an estimate of the form 
\begin{equation}\label{eq:scriptA-Lip-estimate}
    \left\|\mathcal{A}_n(u_1)v-\mathcal{A}_n(u_2)v\right\|_{X_0}\le L^n_\eta \|u_1-u_2\|_{X_p}\|v\|_{X_1}
\end{equation}
with a positive constant $L^n_\eta$ depending at most on $n$, $d$, $\Omega$ and $\eta$. 
To this end, we consider each term in \eqref{def_script_A_n} for fixed $1\leq i,j,k,\ell\leq d$.
We claim that each of the following norms
\begin{align}
&\mathcal N_1:=\left\|(\mu_n\!\left(|D(u_1)|^2)-\mu_n\!\left(|D(u_2)|^2\right)\right)\Delta v_i\right\|_{L^{p}(\Omega)}\nonumber\\
&\mathcal N_2:=\left\|\left(\mu_n\!\left(|D(u_1)|^2\right)-\mu_n\!\left(|D(u_2)|^2\right)\right)\,\partial_{x_i}(\nabla\!\cdot v)\right\|_{L^p(\Omega)}\label{component-of-script-A}\\
&\mathcal N_3:=\left\|\left(\mu_n'\!\left(|D(u_1)|^2\right)
D_{ik}(u_1)D_{j\ell}(u_1)-\mu_n'\!\left(|D(u_2)|^2\right)
D_{ik}(u_2)\,D_{j\ell}(u_2)\right)
\partial_{x_k} D_{j\ell}(v)\right\|_{L^p(\Omega)}\nonumber
\end{align}
 satisfy (up to a multiplicative constant that may depend on $n, d, \Omega$ and $\eta$) an estimate as in \eqref{eq:scriptA-Lip-estimate}. 

We recall that $X_p$ is equipped with the norm $\|\cdot\|_{W^{2-\frac{2}{p},p}(\Omega)}$. 

 We start by noting that 
\begin{equation}
\mathcal N_1,\ \mathcal N_2\le \bigl\|\mu_n(|D(u_1)|^2)-\mu_n(|D(u_2)|^2)\bigr\|_{L^\infty(\Omega)}\|v\|_{X_1}.
\nonumber
\end{equation}
Since $\mu_n:\mathbb{R}^+\to\mathbb{R}$ is Lipschitz continuous on bounded subsets of $\mathbb{R}^+$, there exists a constant $M_n>0$ such that
\[
\begin{aligned}
&\bigl\|\mu_n(|D(u_1)|^2)-\mu_n(|D(u_2)|^2)\bigr\|_{L^\infty(\Omega)}
\le M_{n,\eta}\bigl\||D(u_1)|^2-|D(u_2)|^2\bigr\|_{L^\infty(\Omega)}
\\
&\qquad\le M_{n,\eta} \sum_{i,j=1}^d
\|D_{ij}(u_1)-D_{ij}(u_2)\|_{L^\infty(\Omega)} 
\|D_{ij}(u_1)+D_{ij}(u_2)\|_{L^\infty(\Omega)}.
\end{aligned}
\]
By the embedding \eqref{embedding_W1inf} and the assumption \eqref{assumption_u1_u2<eta}, there exists a constant $C_d>0$ such that
\[
\|D_{ij}(u_1)+D_{ij}(u_2)\|_{L^\infty(\Omega)}\le C_d\,\eta
\quad\text{for all } i,j=1,\dots,d.
\]
Therefore we have
\[
\begin{aligned}
\bigl\|\mu_n(|D(u_1)|^2)-\mu_n(|D(u_2)|^2)\bigr\|_{L^\infty(\Omega)}
&\le
M_{n,\eta} C_d \eta
\sum_{i,j=1}^d
\|D_{ij}(u_1)-D_{ij}(u_2)\|_{L^\infty(\Omega)}
\\
&\le M_{n,\eta} C_d \eta d^2 \norm{u_1-u_2}_{W^{1,\infty}(\Omega)}. 
\end{aligned}\]
Again by \eqref{embedding_W1inf}, 
we conclude that there exists a constant $M_{n,d,\eta}>0$ such that
\begin{equation}\label{estimate_mu_n}
\bigl\|\mu_n(|D(u_1)|^2)-\mu_n(|D(u_2)|^2)\bigr\|_{L^\infty(\Omega)}
\le
M_{n,d,\eta}
\|u_1-u_2\|_{W^{2-\frac{2}{p},p}(\Omega)}.
\end{equation}

Concerning the third norm in \eqref{component-of-script-A}, we immediately see that 
\[\begin{aligned}
\mathcal N_3&\le \bigl\|
\mu_n'\!\left(|D(u_1)|^2\right) D_{ik}(u_1) D_{j\ell}(u_1)
-
\mu_n'\!\left(|D(u_2)|^2\right) D_{ik}(u_2) D_{j\ell}(u_2)
\bigr\|_{L^\infty(\Omega)}\|v\|_{X_1}
\\
&\le\left( 
\bigl\|\mu_n'\!\left(|D(u_1)|^2\right)\bigr\|_{L^\infty(\Omega)}
\bigl\|D_{ik}(u_1)D_{j\ell}(u_1)-D_{ik}(u_2)D_{j\ell}(u_2)\bigr\|_{L^\infty(\Omega)}\right.
\\
&\quad\left.+
\bigl\|D_{ik}(u_2)D_{j\ell}(u_2)\bigr\|_{L^\infty(\Omega)}
\bigl\|\mu_n'\!\left(|D(u_1)|^2\right)
-
\mu_n'\!\left(|D(u_2)|^2\right)\bigr\|_{L^\infty(\Omega)}\right)\|v\|_{X_1}.
\end{aligned}
\]

Note that $\left\|D_{ik}(u_2)D_{j\ell}(u_2)\right\|_{L^\infty(\Omega)}$ is bounded above by $\|u_2\|^2_{W^{1,\infty}(\Omega)}$,  and the latter is bounded by \eqref{embedding_W1inf}. 
 
Similarly, $|D(u_1)|^2$ is bounded and, since $\mu_n'$ is continuous, then $\|\mu_n'\!\left(|D(u_1)|^2\right)\|_{L^\infty(\Omega)}$ is bounded as well.

Since $\mu_n'$ is Lipschitz continuous on bounded subsets of $\R^+$,  we can argue as in the proof of \eqref{estimate_mu_n} and obtain
\[
\bigl\|\mu_n'\!\left(|D(u_1)|^2\right)
-
\mu_n'\!\left(|D(u_2)|^2\right)\bigr\|_{L^\infty(\Omega)}
\le
M'_{n,d,\eta}
\|u_1-u_2\|_{W^{2-\frac{2}{p},p}(\Omega)}
\]
for some $M'_{n,d,\eta}>0$.
Finally, we estimate
\[
\begin{aligned}
\|D_{ik}(u_1)D_{j\ell}(u_1)-D_{ik}(u_2)D_{j\ell}(u_2)\|_{L^\infty(\Omega)}
&\le
\|D_{ik}(u_1)\|_{L^\infty(\Omega)}
\|D_{j\ell}(u_1)-D_{j\ell}(u_2)\|_{L^\infty(\Omega)}
\\
&+
\|D_{j\ell}(u_2)\|_{L^\infty(\Omega)}
\|D_{ik}(u_1)-D_{ik}(u_2)\|_{L^\infty(\Omega)}
\\
&\le
C_{d,\eta}
\|u_1-u_2\|_{W^{2-\frac{2}{p},p}(\Omega)},
\end{aligned}
\]
for some constant $C_{d,\eta}>0$. The last of our claimed bounds is therefore satisfied. 
\end{proof}
\begin{remark}\label{convergence_of_A_n}
    Observe that for all $n \in \mathbb{N} \cup \{\infty\}$ and $s \ge 0$, we have
\[
\mu_n(s) > 0 \quad \text{and} \quad \mu_n(s) + 2 s \, \mu_n'(s) > 0.
\]
Moreover, $\mu_n \to \mu_\infty$ and $\mu_n' \to 0$ uniformly on any compact subset of $[0,\infty)$. 
It then follows immediately that $A_n(u)$ converges to $A_\infty(u)$ in the sense of \ref{uniform_An}, with
\[
A_\infty(u) = -\mu_\infty \Delta.
\]
\end{remark}

We are now ready to complete the proof of the main theorem of this section.

\begin{proof}[Proof of Theorem \ref{th:main-Navier-Stokes}]
By \cite[Theorem 4.1]{Bothe_Pruss}, for all $n \in \mathbb{N} \cup \{\infty\}$ and $u \in X_p$, the operator $A_n(u)$ enjoys the property of maximal $L^p$-regularity. In particular, this property implies that $A_n(u)$ satisfies condition \ref{existence_reason4}.

Moreover, by Lemma \ref{nonlinearity_of_A} and Lemma \ref{Nonlinearity_of_f}, it follows that for each $n \in \mathbb{N} \cup \{\infty\}$ the pair $(A_n, f_n)$ satisfies assumptions \ref{assumption_structure}, \ref{existence_reason1}, \ref{existence_reason2}, and \ref{existence_reason3}.

Finally, in view of Remark \ref{convergence_of_A_n}, condition \ref{uniform_An} is satisfied. Hence, Theorem \ref{main theorem1} applies, and this completes the proof.
\end{proof}

As a consequence of Theorem \ref{th:main-Navier-Stokes}, we can ensure convergence of the corresponding fluid pressure gradients as argued in the next remark. 
\begin{remark}
Equation \eqref{eq:abstract-NS} can be equivalently rewritten as the classical Navier-Stokes equations with an appropriate pressure field $\pi_\infty\in L^p(0,T_\infty;W^{1,p}(\Omega))$ (ensured by the Helmholtz projection):
\begin{equation}\label{eq:NS}
\left\{
\begin{aligned}
   &\nabla \cdot v_\infty = 0
   && \text{in } (0,T)\times \Omega, \\
   &\partial_t (\rho v_\infty)
   -\mu_\infty\Delta v_\infty
   + \nabla \pi_\infty
   = \rho(v_\infty\cdot\nabla)v_\infty 
   && \text{in } (0,T_\infty)\times \Omega, \\
   &v_\infty=0 &&\text{on }(0,T_\infty)\times \Gamma_0,\\
   &\left(v_n \cdot \nu,D(v_\infty)\nu - (D(v_\infty)\nu \cdot \nu)\nu\right)=0&&\text{on }(0,T_\infty)\times \Gamma_s,\\
   &v_\infty(0) = v_0^\infty
   && \text{in } \Omega,
\end{aligned}
\right.
\end{equation}
Let $K$ be a compact subinterval of \( [0, T_\infty) \) with $T_\infty$ as in Theorem~\ref{th:main-Navier-Stokes}, so that 
\[
\lim_{n \to \infty} \|v_n - v_\infty\|_{E_1(K)} = 0.
\]
It follows from Theorem~\ref{Embedding theorem} that
\begin{equation}\label{eq:sup-convergence}
\lim_{n \to \infty} \sup_{t \in K}
\|v_n(t) - v_\infty(t)\|_{X_p} = 0.
\end{equation}
Subtracting side by side \eqref{eq:motion1}$_2$ and \eqref{eq:NS}$_2$, we find that 
\begin{equation}\label{eq:pressure-difference}\begin{split}
\norm{\nabla p_n-\nabla p_\infty}_{L^p(K\times\Omega)}&\le 
\rho\left\|\frac{d v_n}{dt} - \frac{d v_\infty}{dt}\right\|_{E_0(K)}
\\&\quad +\norm{\mathcal A_n(v_n)v_n-\mu_\infty\Delta v_\infty}_{L^p(K\times\Omega)}
\\&\quad+\rho\norm{(v_n\cdot\nabla)v_n-(v_\infty\cdot\nabla)v_\infty}_{L^p(K\times\Omega)}.
\end{split}\end{equation}
For the first term on the right-hand side, we use the definition of the $E_1(K)$-norm (cf. \eqref{def_E1_and_E0_space} and \eqref{def_E1_and_E0}) to get that 
\[
\left\|\frac{d v_n}{dt} - \frac{d v_\infty}{dt}\right\|_{E_0(K)}
\le
\|v_n - v_\infty\|_{E_1(K)}\to 0\quad\text{as }n\to\infty. 
\]
The second term on the right-hand side of \eqref{eq:pressure-difference} converges to zero as $n\to \infty$ thanks to Remark \ref{convergence_of_A_n} and Lemma \ref{nonlinearity_of_A}. Finally, last term in \eqref{eq:pressure-difference} goes to zero as $n\to\infty$ thanks to the estimate \eqref{eq:nonlinear-difference} and the convergence \eqref{eq:sup-convergence}. 
\end{remark}

\appendix
\section{Embedding-type theorem}\label{sec:Embedding theorem}
The following result --stated in a more general form-- can be found in \cite[Appendix B]{navier}. We state it and present its proof in this appendix for completeness. 
\begin{theorem}\label{Embedding theorem}
   For all $t>0$ and $u \in E_1([t_0, t_0+ t])$, there exists $c_1 >0$ (independent of $t$) such that
   \begin{equation}
       \sup_{t \in [t_0, t_0+t]} \norm{ u(t) }_{X_p} \leq c_1\left(\norm{ u(t_0)}_{X_{p}}+\norm{ u }_{E_1([t_0, t_0 + t])}\right)
   \end{equation}
\end{theorem}
\begin{proof}
    Take $u \in E_1([t_0, t_0+ t])$. By \cite[Theorem 4.10.2]{Amann_book_vol_1}, we  know that $u(t_0) \in X_{p}$. Using \cite[Corollary 1.14.(i)]{interpolation_theory}, we find $u_1 \in E_1([t_0, \infty))$  such that $u_1(t_0)=u(t_0)$. Let us consider $v= u-u_1$. It is clear that $v(t_0)=0$, Therefore, there exists $c_0 >0$ (independent of $t$ (see \cite[Proposition 6.1]{pruss_saal_simonett})) such that
   \begin{equation}
       \sup_{t \in [t_0, t_0+t]} \norm{ v(t) }_{X_p} \leq c_0\norm{ v }_{E_1([t_0, t_0 + t])}.
   \end{equation} 
Now, by the triangle inequality, we get \begin{equation}\label{triangle_inequality}
       \sup_{t \in [t_0, t_0+t]} \norm{ u(t) }_{X_p} \leq c_0\left(\norm{ u }_{E_1([t_0, t_0 + t])} + \norm{ u_1 }_{E_1([t_0, t_0 + t])}\right) +  \sup_{t \in [t_0, t_0+t]} \norm{ u_1(t) }_{X_p}.
   \end{equation}  
We know from \cite[Corollary 1.14.(ii)]{interpolation_theory} that 
\begin{equation}\label{embedding_corollary}
\sup_{t \in [t_0, t_0+t]} \norm{ u_1(t) }_{X_p} \leq \sup_{t \in [t_0, \infty]} \norm{ u_1(t) }_{X_p} \leq \frac{p}{p-1}\norm{ u_1 }_{E_1([t_0, \infty))}.    
\end{equation}
Combining \eqref{triangle_inequality} and \eqref{embedding_corollary}, we get that
\begin{equation}\label{before_sup}
\begin{aligned}
    \sup_{t \in [t_0, t_0+t]} \norm{ u(t) }_{X_p}
    \leq c_0\left(\norm{ u }_{E_1([t_0, t_0 + t])} + \norm{ u_1 }_{E_1([t_0, t_0 + t])}\right) + \frac{p}{p-1}\norm{ u_1 }_{E_1([t_0, \infty))}.
    \end{aligned}
\end{equation}
Therefore, by taking the infimum over all \( u_1 \in E_1([t_0, \infty))\) in~\eqref{before_sup} and applying the definition of the \( X_p \)-norm from~\eqref{def_X_p}, we obtain
\begin{equation}\label{after_sup}
\begin{aligned}
    \sup_{t \in [t_0, t_0+t]} \norm{ u(t) }_{X_p}
    \leq c_0\norm{ u }_{E_1([t_0, t_0 + t])} + c_0\norm{ u(t_0)}_{X_{p}} + \frac{p}{p-1}\norm{ u(t_0)}_{X_{p}} 
    \end{aligned}
\end{equation}
Finally, choosing \( c_1 = c_0 + \frac{p}{p-1} \), the theorem follows.
\end{proof}

\section*{Acknowledgments}
The first and third authors gratefully acknowledge the support of the Natural Sciences and Engineering Research Council of Canada (NSERC) through the Discovery Grants RGPIN-2022-04330 and RGPIN-2021-03129.

The third author is a member of the ``Gruppo Nazionale per l'Analisi Matematica, la Probabilit\`a e le loro Applicazioni'' (GNAMPA) of the Istituto Nazionale di Alta Matematica ``Francesco Severi'' (INdAM).

The second and third authors would like to express their gratitude to Gieri Simonett for the enlightening discussions and valuable insights that significantly contributed to this work.

\bibliographystyle{plain}
\bibliography{main}

\end{document}